\documentclass[a4paper,11pt]{article}
\usepackage{amsfonts,latexsym,amsmath,amsthm} 
\usepackage[a4paper,centering,scale=0.7]{geometry}

\newtheorem{proposition}{Proposition}[section]
\newtheorem{theorem}[proposition]{Theorem}
\newtheorem{lemma}[proposition]{Lemma}

\newtheorem{definition}[proposition]{Definition}
\newtheorem{hypothesis}[proposition]{Hypothesis}
\theoremstyle{remark}
\newtheorem{remark}[proposition]{Remark}

\renewcommand{\d}{\mathrm{d}}
\newcommand{\wto}{\rightharpoonup}
\newcommand{\me}{\mathsf{e}}
\newcommand{\mv}{\mathsf{v}}
\newcommand{\E}{\mathbb{E}}

\renewcommand{\P}{\mathbb{P}}
\newcommand{\R}{\mathbb{R}}
\newcommand{\elle}{\mathbb{L}}
\renewcommand{\AA}{\mathcal{A}}
\newcommand{\FF}{\mathcal{F}}
\newcommand{\HH}{\mathcal{H}}
\newcommand{\KK}{\mathcal{K}}
\newcommand{\LL}{\mathcal{L}}
\newcommand{\TT}{\mathcal{T}}

\begin{document}

\title{Stochastic FitzHugh-Nagumo equations on networks\\
with impulsive noise}

\author{Stefano Bonaccorsi\footnote{Dipartimento di Matematica, Universit\`a di Trento, via Sommarive 14, 38100 Povo (TN), Italy. email: bonaccor@science.unitn.it} \and Carlo Marinelli\footnote{Institut
    f\"ur Angewandte Mathematik, Universit\"at Bonn, Wegelerstr. 6, D-53115 Bonn, Germany.} \and Giacomo
  Ziglio\footnote{Dipartimento di Matematica, Universit\`a di Trento, via Sommarive 14, 38100 Povo (TN), Italy. email: ziglio@science.unitn.it}}

\date{July 25, 2008}

\maketitle

\begin{abstract}
  We consider a system of nonlinear partial differential equations
  with stochastic dynamical boundary conditions that arises in models
  of neurophysiology for the diffusion of electrical potentials
  through a finite network of neurons. Motivated by the discussion in
  the biological literature, we impose a general diffusion equation on
  each edge through a generalized version of the FitzHugh-Nagumo
  model, while the noise acting on the boundary is described by a
  generalized stochastic Kirchhoff law on the nodes. In the abstract
  framework of matrix operators theory, we rewrite this stochastic
  boundary value problem as a stochastic evolution equation in
  infinite dimensions with a power-type nonlinearity, driven by an
  additive L\'evy noise. We prove global well-posedness in the mild
  sense for such stochastic partial differential equation by
  monotonicity methods.
\end{abstract}

\section{Introduction}
In this paper we study a system of nonlinear diffusion equations on a
finite network in the presence of an impulsive noise acting on the
nodes of the system.  We allow a rather general nonlinear drift term
of polynomial type, including functions of FitzHugh-Nagumo type
(i.e. $f(u) = -u(u-1)(u-a)$) arising in various models of
neurophysiology (see e.g.  the monograph \cite{KeeSne} for more
details).

Electric signaling by neurons has been studied since the 50s, starting
with the now classical Hodgkin-Huxley model \cite{HoHu} for the
diffusion of the transmembrane electrical potential in a neuronal
cell. This model consists of a system of four equations describing the
diffusion of the electrical potential and the behaviour of various ion
channels.  Successive simplifications of the model, trying to capture
the key phenomena of the Hodgkin-Huxley model, lead to the reduced
FitzHugh-Nagumo equation, which is a scalar equation with three stable
states (see e.g. \cite{Giugi}).

Among other papers dealing with the case of a whole neuronal network
(usually modeled as a graph with $m$ edges and $n$ nodes), which is
intended to be a simplified model for a large region of the brain, let
us mention a series of recent papers by Mugnolo et al.
\cite{KrMuSi,MuRo}, where the well-posedness of the isolated system is
studied.

Note that, for a diffusion on a network, other conditions must be
imposed in order to define the behaviour at the nodes. 
We impose a continuity condition, that is, given any node in the
network, the electrical potentials of all its incident edges are equal.
Each node represents an active soma, and in this part of the cell
the potential evolves following a generalized Kirchhoff condition that
we model with dynamical boundary conditions for the internal dynamics.

Since the classical work of Walsh \cite{Walsh}, stochastic partial
differential equations have been an important modeling tool in
neurophysiology, where a random forcing is introduced to model
external perturbations acting on the system. In our neuronal network,
we model the electrical activity of background neurons with a
stochastic input of impulsive type, to take into account the stream of
excitatory and inhibitory action potentials coming from the neighbors
of the network. The need to use models based on impulsive noise was
already pointed out in several papers by Kallianpur and coauthors --
see e.g.  \cite{KaWo,KaXi}. On the other hand, from a mathematical
point of view, the addition of a Brownian noise term does not affect
the difficulty of the problem. In fact, in section \ref{sec:set}
below, a Wiener noise could be added taking $q\neq0$, introducing
an extra term that does not modify the estimates obtained in section
$\ref{sec:3}$, which are the basis for the principal results of this
paper.  Let us also recall that the existence and uniqueness of
solutions to reaction-diffusion equations with additive Brownian noise
is well known -- see e.g. \cite{cerrai-libro,DP-K,DZ96}.

Following the approach of \cite{BoZi}, we use the abstract setting of
stochastic PDEs by semigroup techniques (see e.g. \cite{DP-K,DZ92}) to
prove existence and uniqueness of solutions to the system of
stochastic equations on a network.  In particular, the specific
stochastic dynamics is rewritten in terms of a stochastic evolution
equation driven by an additive L\'evy noise on a certain class of
Hilbert spaces. 

The rest of the paper is organized as follows: in section
\ref{sec:set} we introduce the problem and we motivate our assumptions
in connection with the applications to neuronal networks.  Then we
provide a suitable abstract setting and we prove, following
\cite{MuRo}, that the linear operator appearing as leading drift term
in the stochastic PDE generates an analytic semigroup of contractions.
Section \ref{sec:3} contains our main results.  First we prove
existence and uniqueness of mild solution for the problem under
Lipschitz conditions on the nonlinear drift term. This result (essentially
already known) is used to obtain existence and uniqueness in the mild
sense for the SPDE with a locally Lipschitz drift of FitzHugh-Nagumo
type by monotonicity techniques.

\section{Setting of the problem}
\label{sec:set}
Let us begin introducing some notation used throughout the paper. We
shall denote by $\wto$ and $\stackrel{\ast}{\wto}$, respectively, weak
and weak* convergence of functions. All stochastic elements are
defined on a (fixed) filtered probability space
$(\Omega,\mathcal{F},\mathcal{F}_t,\P)$ satisfying the usual
hypotheses.  Given a Banach space $E$, we shall denote by
$\mathbb{L}^p(E)$ the space of $E$-valued random variables with finite
$p$-th moment.

The network is identified with the underlying graph $G$,
described by a set of
$n$ vertices $\mv_1, \dots, \mv_n$ and $m$ oriented edges
$\me_1,\dots, \me_m$ which we assume to be normalized, i.e., $\me_j =
[0,1]$. The graph is described by the {\em incidence matrix} $\Phi =
\Phi^+ - \Phi^-$, where $\Phi^+ = (\phi_{ij}^+)_{n \times m}$ and
$\Phi^- = (\phi_{ij}^-)_{n \times m}$ are given by
\begin{equation*}
  \phi^-_{ij} = 
  \begin{cases}
    1, & \mv_i=\me_j(1) \\ 0, & \text{otherwise}
  \end{cases}
  \qquad
  \phi^+_{ij} = 
  \begin{cases}
    1, & \mv_i=\me_j(0) \\ 0, & \text{otherwise.}
  \end{cases}
\end{equation*}
The degree of a vertex is the number of edges entering or leaving the
node. We denote
\begin{equation*}
  \Gamma(\mv_i) = \{j \in \{1,\dots,m\}\ :\ \me_j(0) = \mv_i \text{ or
  } \me_j(1) = \mv_i\}
\end{equation*}
hence the degree of the vertex $\mv_i$ is the cardinality
$|\Gamma(\mv_i)|$. 

The electrical potential in the network shall be denoted by
$\bar{u}(t,x)$ where $\bar{u} \in (L^2(0,1))^m$ is the vector
$(u_1(t,x),\dots, u_m(t,x))$ and $u_j(t,\cdot)$ is the electrical
potential on the edge $\me_j$. We impose a general diffusion equation
on every edge
\begin{equation}
  \label{eq:E}
  \frac{\partial}{\partial t}u_j(t,x) = \frac{\partial}{\partial
    x}\left(c_j(x)\frac{\partial}{\partial x}u_j(t,x)\right) +
  f_j(u_j(t,x)),
\end{equation}
for all $(t,x)\in \R_+ \times (0,1)$ and all $j=1,...,m$.  The
generality of the above diffusion is motivated by the discussion in
the biological literature, see for example \cite{KeeSne}, who remark, in
discussing some concrete biological models, that the basic cable
properties is not constant throughout the dendritic tree.  The
above equation shall be endowed with suitable boundary and initial
conditions. Initial conditions are given for simplicity at time $t=0$
of the form
\begin{equation}
  \label{eq:I}
  u_j(0,x) =u_{j0}(x) \in C([0,1]),\qquad j=1,...,m.
\end{equation}
Since we are dealing with a diffusion in a
network, we require first a continuity assumption on every node
\begin{equation}
  \label{eq:C}
  p_i(t) := u_j(t,\mv_i)=u_k(t,\mv_i),\qquad
  t>0,\ j,k\in\Gamma(\mv_i),\ i=1,...,n 
\end{equation}
and a stochastic generalized Kirchhoff law in the nodes
\begin{equation}
  \label{eq:D}
  \frac{\partial}{\partial t} p_i(t) = -b_i p_i(t) + \sum_{j \in
    \Gamma(\mv_i)} \phi_{ij} \mu_j c_j(\mv_i) \frac{\partial}{\partial
    x}u_j(t,\mv_i) + \sigma_i \frac{\partial}{\partial t} L(t,\mv_i),
\end{equation}
for all $t>0$ and $i=1,\ldots,n$.
Observe that the plus sign in front of the Kirchhoff term in the above
condition is consistent with a model of purely excitatory node
conditions, i.e. a model of a neuronal tissue where all synapses
depolarize the postsynaptic cell. Postsynaptic potentials can have
graded amplitudes modeled by the constants $\mu_j > 0$ for all $j=1,...,m$.

Finally, $L(t,\mv_i), \ i=1,...,n$, represent the stochastic
perturbation acting on each node, due to the external surrounding, and
$\tfrac{\partial}{\partial t} L(t,\mv_i)$ is the formal time
derivative of the process $L$, which takes a meaning only in integral
sense. Biological motivations lead us to model this term by a
L\'evy process. In fact, the evolution of the electrical
potential on the molecular membrane can be perturbed by different
types of random terms, each modeling the influence, at different time
scale, of the surrounding medium. On a fast time scale, vesicles of
neurotransmitters released by external neurons cause electrical
impulses which arrive randomly at the soma causing a sudden change in
the membrane voltage potential of an amount, either positive or
negative, depending on the composition of the vesicle and possibly
even on the state of the neuron. We model this behaviour perturbing the
equation by an additive $n$-dimensional impulsive
noise of the form
\begin{equation}
  \label{eq:5}
  L(t)= \int_{\R^n} x \tilde{N}(t,{\rm d} x).
\end{equation}
See Hypothesis \ref{hp2} below for a complete description of the
process and \cite{KaXi} for a related model.

Although many of the above reasonings remain true also when considering
the diffusion process on the fibers, we shall not pursue such
generality and assume that the random perturbation acts only on the
boundary of the system, i.e. on the nodes of the network.

\medskip

Let us state the main assumptions on the data of the problem.
\begin{hypothesis}\label{hp1}
  \begin{enumerate}
  \item[]
  \item\label{hp1:1} In (\ref{eq:E}), we assume that $c_j(\cdot)$ belongs
    to $C^1([0,1])$, for $j=1,\dots, m$ and $c_j(x) > 0$ for every $x \in
    [0,1]$.
  \item\label{hp1:2} There exists constants $\eta \in \R$, $c > 0$
    and $s \ge 1$ such that, for $j=1,\dots,m$, the functions $f_j(u)$
    satisfy $f_j(u) + \eta u$ is continuous and decreasing, and
    $|f_j(u)| \le c (1+|u|^s)$.
  \item\label{hp1:3} In (\ref{eq:D}), we assume that $b_i \ge 0$ for every $i =
    1,\dots,n$ and at least one of the coefficients $b_i$ is strictly
    positive.
  \item $\{\mu_j\}_{j=1,...,m}$ and $\{\sigma_i\}_{i=1,\ldots,n}$ are
    real positive numbers.
  \end{enumerate}
\end{hypothesis}

Given a Hilbert space $\mathcal{H}$, let us define the space
$L^2_\mathcal{F}(\Omega\times[0,T];\HH)$ of adapted processes $Y:[0,T]
\to \HH$ endowed with the natural norm
\[
|Y|_2=\left(\E\int_0^T |Y(t)|^2_\HH\d t\right)^{1/2}.
\]
We shall consider a L\'evy process $\{L(t),\ t \ge 0\}$ with values in
$(\R^n,{\mathcal B}(\R^n))$, i.e., a stochastically continuous,
adapted process starting almost surely from 0, with stationary
independent increments and c\`adl\`ag trajectories. By the classical
L\'evy-It\^o decomposition theorem, one has
\begin{equation}\label{Levy-dec}
  L(t) = m t + q W_t + \int_{|x| \le 1} x [N(t,{\rm d}x) -t \nu({\rm d}
  x)]+\int_{|x| > 1} x N(t,{\rm d} x),\qquad t\geq 0
\end{equation}
where $m\in\R^n$, $q\in M_{n\times n}(\R)$ is a symmetric, positive
defined matrix, $\{W_t,\ t \ge 0\}$ is an $n$-dimensional centered
Brownian motion, $N(t,\d x)$ is a Poisson measure and the L\'evy measure
$\nu(\d x)$ is $\sigma$-finite on $\R^n\setminus\{0\}$ and such that
$\int\min(1,x^2)\nu(\d x)<\infty$.  We denote by $\tilde{N}(\d t,\d
x):=N(\d t,\d x) -\d t \nu(\d x)$ the compensated Poisson measure.

\begin{hypothesis}\label{hp2}
  We suppose that the measure $\nu$ has finite second order moment, i.e.
  \begin{equation}\label{cond1}
    \int_{\R^n} |x|^2 \nu({\rm d}x) < \infty.
  \end{equation}
\end{hypothesis}
Condition (\ref{cond1}) implies that the generalized compound
Poisson process $\int_{|x|>1} x \, N(t,{\rm d} x)$ has finite moments
of first and second order. Then, with no loss of generality, we assume
that
\begin{equation}
  \label{eq:cond1-bis}
  \int_{|x|>1} x \nu({\rm d}x) = 0.
\end{equation}
We also assume throughout that the L\'evy process is a pure jump
process, i.e. $m\equiv 0$ and $q \equiv 0$, which leads to the
representation \eqref{eq:5} in view of assumptions (\ref{cond1}) and \eqref{eq:cond1-bis}.

\subsection{Well-posedness of the linear deterministic problem}

We consider the product space ${\mathbb H}=(L^2(0,1))^m$. A vector
$\bar u \in {\mathbb H}$ is a collection of functions $\{u_j(x),\ x
\in [0,1],\ j=1,\dots, m\}$ which represents the electrical potential
inside the network.

\begin{remark}
  For any real number $s \ge 0$ we define the Sobolev spaces
  \begin{equation*}
    {\mathbb H}^s =(H^s(0,1))^m,
  \end{equation*}
  where $H^s(0,1)$ is the fractional Sobolev space defined for
  instance in \cite{LiMa}.  In particular we have that ${\mathbb H}^1
  \subset (C[0,1])^m$.  Hence we are allowed to define the boundary
  evaluation operator $\Pi: {\mathbb H}^1 \to \R^n$ defined by
  \begin{equation*}
    \Pi \bar u = 
    \begin{pmatrix}
      p_1 \\ \vdots \\ p_n
    \end{pmatrix},\quad\mbox{where }
    p_i = \bar u(\mv_i) = u_k(\mv_i) \quad \text{for $k \in \Gamma(\mv_i),\;i=1,...,n$.}
  \end{equation*}
\end{remark}

On the space ${\mathbb H}$ we introduce the linear operator $(A,D(A))$
defined by
\begin{equation*}
  \begin{aligned}
    &D(A) = \{\bar{u} \in {\mathbb H}^2 \mid \exists\, p \in \R^n
    \text{ such that } \Pi \bar u = p\}
    \\
    &A \bar{u} = \left( \phantom{\frac{\partial}{\partial}}\frac{\partial}{\partial x}
      \left(c_j(x)\frac{\partial}{\partial x}u_j(t,x)\right)
      \right)_{j=1, \dots, m}
  \end{aligned}
\end{equation*}

As discussed in \cite{MuRo}, the diffusion operator $A$ on a
network, endowed with active nodes, fits the abstract mathematical
theory of parabolic equations with dynamic boundary conditions, and in
particular it can be discussed in an efficient way by means of
sesquilinear forms.

Notice that no other condition except continuity on the nodes
is imposed on the elements of $D(A)$.  This is often stated by saying
that the domain is {\em maximal}.

The so called feedback operator, denoted by $C$, is a linear operator
from $D(A)$ to $\R^n$ defined as
\begin{equation*}
  C \bar{u} = \left( \sum_{j \in \Gamma(\mv_i)} \phi_{ij} \mu_j
    c_j(\mv_i) \frac{\partial}{\partial
      x}u_j(t,\mv_i)\right)_{i=1,\dots,n}.
\end{equation*}
On the vector space $\R^n$ we also define the diagonal matrix
\begin{equation*}
  B = 
  \begin{pmatrix}
    -b_1 \\ & \ddots \\ & & -b_n
  \end{pmatrix}.
\end{equation*}

\smallskip 

With the above notation, problem \eqref{eq:E}--\eqref{eq:D} can be
written as an abstract Cauchy problem on the product space $\HH = {\mathbb H} \times \R^n$ endowed with the natural inner product
\begin{equation*}
  \langle X, Y \rangle_{{\mathcal H}} = \langle \bar u, \bar v
  \rangle_{{\mathbb H}} + \langle p,q\rangle_{\R^n},
  \qquad\mbox{where }X,Y\in\HH\mbox{ and } X = \begin{pmatrix}
    \bar u \\ p \end{pmatrix},\ 
  Y = \begin{pmatrix}
    \bar v \\ q \end{pmatrix}.
\end{equation*}
We introduce the matrix operator ${\mathcal A}$ on the space $\HH$, given in
the form
\begin{equation*}
  {\mathcal A} =
  \begin{pmatrix}
    A & 0 \\ C & B
  \end{pmatrix}
\end{equation*}
with domain
\begin{equation*}
  D({\mathcal A}) = \{X = (\bar{u},p) \in\HH \,:\, \bar{u} \in D(A),
  u_j(\mv_i) = p_i \ \text{for every $j \in \Gamma(\mv_i)$}\}.
\end{equation*}
Then the linear deterministic part of problem
\eqref{eq:E}--\eqref{eq:D} becomes
\begin{equation}
  \label{eq:acp}
  \left\{
  \begin{aligned}
    \frac{{\rm d}}{{\rm d} t} X(t) &= {\mathcal A}X(t)  \\
    X(0)&= x_0
  \end{aligned}
  \right.
\end{equation}
where $x_0=(u_j(0,x))_{j=1,...,m}\in C([0,1])^m$ is the vector of initial conditions.
This problem is well posed, as the following result shows.

\smallskip
\begin{proposition}\label{generation}
  Under Hypotheses \ref{hp1}.\ref{hp1:1} and \ref{hp1}.\ref{hp1:2} 
  the operator $({\mathcal A},D({\mathcal A}))$ is self-adjoint,
  dissipative and has compact resolvent. In particular, it
  generates a $C_0$ analytic semigroup of contractions.
\end{proposition}

\smallskip
\begin{proof}
  For the sake of completeness, we provide a sketch of the proof
  following \cite{MuRo}. The idea is simply to associate the
  operator $({\mathcal A},D({\mathcal A}))$ with a suitable form
  ${\mathfrak a}(X,Y)$ having dense domain ${\mathcal V} \subset
  {\mathcal H}$.

  The space ${\mathcal V}$ is defined as
  \begin{equation*}
    {\mathcal V} = \left\{ X = 
      \begin{pmatrix}
        \bar{u} \\ p
      \end{pmatrix} \mid \bar{u} \in (H^1(0,1))^m,\ u_k(\mv_i) = p_i\ \text{for $i=1,\dots,n$, $k \in \Gamma(\mv_i)$}\right\}
  \end{equation*}
  and the form ${\mathfrak a}$ is defined as
  \begin{equation*}
    {\mathfrak a}(X,Y) = \sum_{j=1}^m \int_0^1 \mu_j c_j(x) u'_j(x)
    v'_j(x) \, {\rm d}x + \sum_{l=1}^n b_l p_l q_l, \qquad X =
    \begin{pmatrix} \bar u \\ p
  \end{pmatrix},\ Y = 
  \begin{pmatrix}
    \bar v \\ q
  \end{pmatrix}.
  \end{equation*}
  The form ${\mathfrak a}$ is clearly positive and symmetric;
  furthermore it is closed and continuous. Then a little computation
  shows that the operator associated with ${\mathfrak a}$ is
  $({\mathcal A},D({\mathcal A}))$ defined above. Classical
  results in Dirichlet forms theory, see for instance \cite{ouhabaz},
  lead to the desired result.
\end{proof}

The assumption that $b_l > 0$ for some $l$ is a dissipativity
condition on ${\mathcal A}$. In particular it implies the following
result (for a proof see \cite{MuRo}).
\smallskip
\begin{proposition}
  Under Hypotheses \ref{hp1}.\ref{hp1:1} and \ref{hp1}.\ref{hp1:3},
  the operator ${\mathcal A}$ is invertible and the semigroup
  $\{{\mathcal T}(t),\ t \ge 0\}$ generated by ${\mathcal A}$ is
  exponentially bounded, with growth bound given by the strictly
  negative spectral bound of the operator ${\mathcal A}$.
\end{proposition}

\section{The stochastic Cauchy problem}
\label{sec:3}
We can now solve the system of stochastic differential equations
(\ref{eq:E})--~(\ref{eq:D}).  The functions $f_j(u)$ which appear in
(\ref{eq:E}) are assumed to have a polynomial growth.  We remark that
the classical FitzHugh-Nagumo problem requires
$$
f_j(u)= u(u-1)(a_j-u)\qquad j=1,...,m
$$
for some $a_j \in (0,1)$, and satisfies Hypothesis
\ref{hp1}.\ref{hp1:2} with
\[
\eta \leq -\max_j \frac{(a_j^3+1)}{3(a_j+1)}, \qquad s=3.
\]
We set
\begin{equation}\label{eq:F_fitz-nag}
  F(\bar u) = 
  \begin{pmatrix}
    f_j(u_j)
  \end{pmatrix}_{j=1,\dots,m} \quad \text{ and }\quad
\mathcal{F}(X)=\begin{pmatrix}-F({\bar u})\\0\end{pmatrix}\quad\text{ for }X=\begin{pmatrix}{\bar u}\\p\end{pmatrix},
\end{equation}
and we write our problem in abstract form
\begin{equation}\label{eq:NL}
\left\{
  \begin{aligned}
    {\rm d} X(t) &= [\AA X(t) - \mathcal{F}(X(t))] \, {\rm d}t+ \Sigma \, {\rm
      d}\LL(t)
    \\
    X(0)&= x_0,
  \end{aligned}
  \right.
\end{equation}
where $\Sigma$ is the matrix defined by
\[
\Sigma = \begin{pmatrix}
0 & 0 \\ 0 & \sigma
\end{pmatrix}
=
\begin{pmatrix}
0 & 0 \\ 0 & {\rm diag}(\sigma_1,\dots,\sigma_n)
\end{pmatrix},
\]
and $\LL(t)$ is the natural embedding in $\HH$
of the $n$-dimensional L\'evy process $L(t)$, i.e.
\begin{equation*}
  \LL(t) = 
  \begin{pmatrix}
    0 \\ L(t)
  \end{pmatrix}.
\end{equation*}
\begin{remark}
  Note that in general $\FF$ is only defined on its domain $D(\FF)$, which is
  strictly smaller than $\HH$.
\end{remark}
Let us recall the definition of mild solution for the stochastic
Cauchy problem \eqref{eq:NL}.

\begin{definition}
  \label{de:mild-solution}
  An $\HH$-valued predictable process $X(t)$, $t\in[0,T]$, is said to
  be a \emph{mild solution} of (\ref{eq:NL}) if
  \begin{equation}\label{cond:mild}
  \int_0^T|\mathcal{F}(X(s))|\,\d s<+\infty
  \end{equation}
  and
  \begin{equation}\label{eq:mild}
    X(t) = \TT(t)x_0 - \int_0^t \TT(t-s)\mathcal{F}(X(s)) \,\d s +
    \int_0^t\TT(t-s)\Sigma\, \d\LL(s)
  \end{equation}
  $\mathbb{P}$-a.s. for all $t\in[0,T]$, where $\TT(t)$ is the semigroup
  generated by $\AA$.
\end{definition}
Condition (\ref{cond:mild}) implies that the first integral in
(\ref{eq:mild}) is well defined. The second integral, which we shall
refer to as stochastic convolution, is well defined as will be shown
in the following subsection.

\subsection{The stochastic convolution process}
\label{se:levy}
In our case the stochastic convolution can be written as
\begin{equation*}
  Z(t) = \int_0^t \int_{\R^n} \TT(t-s) \begin{pmatrix}0 \\ \sigma
    x\end{pmatrix} \, \tilde{N}({\rm d}s,{\rm d} x).
\end{equation*}
The definition of stochastic integral with respect to a compensated
Poisson measure has been discussed by many authors, see for instance
\cite{AlbRud-LI,App2,App1,choj,GS-III,Hau}. Here we limit ourselves to
briefly recalling some conditions for the existence of such integrals. In
particular, in this paper we only integrate deterministic functions,
such as $\TT(\cdot)\Sigma$, taking values in (a subspace of) $L(\HH)$,
the space of linear operators from $\HH$ to $\HH$. In order to define
the stochastic integral of this class of processes with respect to the
L\'evy martingale-valued measure
\begin{equation}\label{martingale-measure} 
   M(t,B)=\int_B x\, \tilde{N}(t,{\rm d} x),
\end{equation}
one requires that the mapping
$\TT(\cdot)\Sigma:[0,T]\times\R^n\ni(t,x)\mapsto\TT(t)(0,\sigma x)$
belongs to the space $L^2((0,T)\times B;\langle M({\rm d} t,{\rm d}
x)\rangle)$ for every $B\in\mathcal{B}(\R^n)$, i.e. that
\begin{equation}\label{strongintegral}
    \int_0^T \int_B \left|\TT(s)
        \begin{pmatrix}0\\\sigma x\end{pmatrix}
      \right|_\HH^2 \, \nu({\rm d} x)\, {\rm d} s< \infty.
  \end{equation}
Thanks to (\ref{cond1}), one has
\begin{align*}
    &\int_0^T \int_B \left |\TT(s)
      \begin{pmatrix}0\\\sigma x\end{pmatrix} \right|_\HH^2 \,
    \nu({\rm d} x)\, {\rm d} s\\
    &\qquad \leq |\sigma|^2 \left(\int_0^T |\TT(s)|_{L(\HH)}^2 \, {\rm d}
      s\right) \left(\int_B |x|^2 \, \nu({\rm d} x)\right) < \infty,
\end{align*}
thus the stochastic convolution $Z(t)$ is well defined for all $t\in[0,T]$.

We shall now prove a regularity property (in space) of the stochastic
convolution. Below we will also
see that the stochastic convolution has c\`adl\`ag paths.

Let us define the product spaces $\mathcal{E}:=(C[0,1])^m\times\R^n$
and $C_\mathcal{F}([0,T];L^2(\Omega;\mathcal{E}))$, the space of
$\mathcal{E}$-valued, adapted mean square continuous processes $Y$ on
the time interval $[0,T]$ such that
\[
|Y|^2_{C_\mathcal{F}} := \sup_{t\in[0,T]}\E|Y(t)|_\mathcal{E}^2 
< \infty.
\]
\begin{lemma}\label{lemma1}
  For all $t \in[0,T]$, the stochastic convolution $\{Z(t),\ t \in
  [0,T]\}$ belongs to the space $C_\mathcal{F}([0,T];L^2(\Omega;{\mathcal
    E}))$. In particular, $Z(t)$ is predictable.
\end{lemma}
\begin{proof}
Let us recall that the (unbounded) matrix operator $\AA$ on $\HH$ is
defined by
\[
\AA=\begin{pmatrix}\partial_x^2&0\\-\partial_\nu&B\end{pmatrix}
\]
with domain $D({\mathcal A}) = \{X = (\bar{u},p) \in\HH \,:\,
\bar{u} \in D(A), u_l(\mv_i) = p_i \ \text{for every $l \in
  \Gamma(\mv_i)$}\}$, and, by proposition \ref{generation}, it
generates a $C_0$-analytic semigroup of contractions on $\HH$.

Let us introduce the interpolation spaces
$\HH_\theta=(\HH,D(\AA))_{\theta,2}$ for $\theta\in(0,1)$.  By
classical interpolation theory (see e.g. \cite{Lu}) it results that,
for $\theta<1/4$, $\HH_\theta=\mathbb{H}^{2\theta}\times\R^n$ while
for $\theta>1/4$ the definition of $\HH_\theta$ involves boundary
conditions, that is
$$
\HH_\theta=\left\{\begin{pmatrix}\bar{u}\\p\end{pmatrix}\in H^{2\theta}\times\R^n:\Pi\bar{u}=p\right\}.
$$
Therefore, one has $(0,\sigma x)\in\HH_\theta$ for
$\theta<1/4$. Furthermore, for $\theta>1/2$, one also has
$\HH_\theta\subset \mathbb{H}^1\times\R^n\subset (C[0,1])^m\times\R^n$ by
Sobolev embedding theorem. Moreover, for all $x\in\HH_\theta$ and
$\theta+\gamma\in(0,1)$, it holds
$$
|\TT(t)x|_{\theta+\gamma}\leq t^{-\gamma}|x|_\theta e^{\omega_{\AA} t},
$$  
where $\omega_{\AA}$ is the spectral bound of the operator $\AA$.  

Let $\theta,\gamma$ be real numbers such that $\theta\in(0,1/4)$,
$\gamma\in(0,1/2)$ and $\theta+\gamma\in(1/2,1)$. Then for all
$t\in[0,T]$
\[
|Z(t)|_{\theta+\gamma}\leq
\int_0^t\int_{\R^n}\left|\TT(t-s)
\begin{pmatrix}0\\\sigma x\end{pmatrix}
\right|_{\theta+\gamma}\tilde{N}(\d x,\d s)
\qquad \mathbb{P}\text{-a.s.}
\]
The right hand side of the above inequality is well defined if and
only if
\[
\E\left|\int_0^T\int_{\R^n}\left|\TT(s)\begin{pmatrix}0\\\sigma x\end{pmatrix}\right|_{\theta+\gamma}\tilde{N}(\d x,\d s)\right|^2
=
\int_0^T\int_{\R^n}\left|\TT(s)\begin{pmatrix}0\\\sigma x\end{pmatrix}\right|^2_{\theta+\gamma}\nu(\d x)\d s
<\infty,
\]
where the identity follows by the classical isometry for Poisson integrals.
On the other hand, one has
\begin{eqnarray*}
\int_0^T\int_{\R^n}\left|\TT(s)\begin{pmatrix}0\\\sigma x\end{pmatrix}\right|^2_{\theta+\gamma}\nu(\d x)\d s
&\leq&\int_0^T\int_{\R^n}s^{-2\gamma}\left|\begin{pmatrix}0\\\sigma x\end{pmatrix}\right|^2_{\theta}e^{2\omega_\AA s}\nu(\d x)\d s\\
&\leq&|\sigma|^2\int_0^Ts^{-2\gamma}e^{2\omega_\AA s}\d s\int_{\R^n}|x|^2\nu(\d x)<\infty
\end{eqnarray*}
using $\gamma\in(0,1/2)$ and assumption (\ref{cond1}). So
$Z(t)\in\HH_{\theta+\gamma}$ for $\theta+\gamma>1/2$ and then $Z(t)\in
(C[0,1])^m\times\R^n=\mathcal{E}$.  It remains to prove that $Z(t)$ is mean
square continuous as $\mathcal{E}$-valued process.  For $0 \le s < t \le T$ we
can write
\begin{eqnarray*}
\E|Z(t)-Z(s)|_\mathcal{E}^2&=& \E\left|\int_0^t \TT(t-r)\Sigma \,\d\LL(r) - \int_0^s\TT(s-r)\Sigma \d\LL(r)\right|_\mathcal{E}^2\\
&\leq& 2\E\left| \int_0^s\int_{\R^n} [\TT(t-r)-\TT(s-r)] \begin{pmatrix} 0 \\ \sigma x\end{pmatrix} \, \tilde{N}(\d  x, \d r)\right|_\mathcal{E}^2\\
&& + 2 \E\left|\int_s^t\int_{\R^n} \TT(t-r) \begin{pmatrix} 0 \\ \sigma x\end{pmatrix} \, \tilde{N}(\d x,\d r)\right|_\mathcal{E}^2\\
&=&2 \int_0^s\int_{\R^n}\left|[\TT(t-r)-\TT(s-r)] \begin{pmatrix} 0 \\ \sigma x\end{pmatrix}\right|_\mathcal{E}^2 \, \nu(\d x) \d r\\
&&+2\int_s^t\int_{\R^n}\left|\TT(t-r) \begin{pmatrix} 0 \\ \sigma x\end{pmatrix}\right|_\mathcal{E}^2 \, \nu(\d x)\d r\longrightarrow 0
\end{eqnarray*}
by the strong continuity of the semigroup $\TT(t)$.  Since the
stochastic convolution $Z(t)$ is adapted and mean square continuous,
it is predictable.
\end{proof}

\subsection{Existence and uniqueness in the Lipschitz case}
We consider as a preliminary step the case of Lipschitz continuous
nonlinear term and we prove existence and uniqueness of solutions in
the space $C_\mathcal{F}$ of adapted mean square continuous processes
taking values in $\mathcal{H}$. We would like to mention that this
result is included only for the sake of completeness and for the
simplicity of its proof (which is essentially based only on the isometry
defining the stochastic integral). In fact, a much more general
existence and uniqueness result was proved by Kotelenez in
\cite{Kote-Doob}.

\begin{theorem}     \label{th:caso-lip}
  Assume that Hypothesis \ref{hp2} holds, and let $x_0$ be an
  ${\mathcal F}_0$-measurable $\HH$-valued random variable such that
  $\E|x_0|^2<\infty$.
  Let $G:\HH\to\HH$ be a function satisfying Lipschitz and linear
  growth conditions:
  \begin{equation}
    \label{cond:lip}
    |G(x)|\leq c_0(1+|x|),\qquad|G(x) - G(y)| \le c_0 |x-y|,\qquad x,y \in \HH.
  \end{equation}
  for some constant $c_0>0$. Then there exists a unique mild solution
  $X\in C^0([0,T];L^2(\Omega,\HH))$ to equation (\ref{eq:NL}) with
  $-\FF$ replaced by $G$. Moreover, the solution map $x_0 \mapsto
  X(t)$ is Lipschitz continuous.
\end{theorem}
\begin{proof}
  We follow the semigroup approach of \cite[Theorem 7.4]{DZ92} where
  the case of Wiener noise is treated. We emphasize only the main
  differences in the proof.
  
  The uniqueness of solutions reduces to a simple application of
  Gronwall's inequality.  To prove existence we use the classical
  Banach's fixed point theorem in the space $C_{\mathcal
    F}([0,T];L^2(\Omega;{\mathcal H}))$.  Let $\KK$ be the mapping
  \begin{equation*}
    \KK(Y)(t) = \TT(t)x_0 + \int_0^t\TT(t-s)G(Y(s))\, {\rm d} s +
    Z(t)
  \end{equation*}
  where $Y\in C_{\mathcal F}([0,T];L^2(\Omega;{\mathcal H}))$ and
  $Z(t)$ is the stochastic convolution.  $Z(\cdot)$ and
  $\TT(\cdot)x_0$ belong to $C_{\mathcal F}([0,T];L^2(\Omega;{\mathcal
    H}))$ respectively in view of Lemma \ref{lemma1} and the assumption on $x_0$. Moreover, setting
  \begin{equation*}
    \KK_1(Y)(t) = \int_0^t\TT(t-s)G(Y(s))\, {\rm d} s,
  \end{equation*}
  it is sufficient to note that
  \begin{equation*}
    |\KK_1(Y)|_{C_\mathcal{F}}^2\leq (Tc_0)^2(1+|Y|^2_{C_\mathcal{F}})
  \end{equation*}
  by the linear growth of $G$ and the contractivity of $\TT(t)$.
Then we obtain that $\KK$ maps the space $C_{\mathcal F}([0,T];L^2(\Omega;{\mathcal H}))$ to itself.
  Furthermore, using the Lipschitz continuity of $G$, it follows that for
  arbitrary processes $Y_1$ and $Y_2$ in  $C_{\mathcal
    F}([0,T];L^2(\Omega;{\mathcal H}))$ we have
  \begin{equation*}
    |\KK(Y_1) - \KK(Y_2)|_{C_\mathcal{F}}^2 = |\KK_1(Y_1)-\KK_1(Y_2)|_{C_\mathcal{F}}^2 \leq
    (c_0T)^2|Y_1-Y_2|_{C_\mathcal{F}}^2.
  \end{equation*}
  If we choose an interval $[0,\tilde{T}]$ such that $\tilde{T} <
  c_0^{-1}$, it follows that the mapping $\KK$ has a unique fixed
  point $X \in C_{\mathcal F}([0,\tilde T];L^2(\Omega;{\mathcal H}))$.
  The extension to an arbitrary interval $[0,T]$ follows by patching
  together the solutions in successive time intervals of length
  $\tilde{T}$.

  The Lipschitz continuity of the solution map $x_0 \mapsto X$ is
  again a consequence of Banach's fixed point theorem, and the proof
  is exactly as in the case of Wiener noise.

  It remains to prove the mean square continuity of $X$. Observe that
  $\TT(\cdot)x_0$ is a deterministic continuous function and it
  follows, again from Lemma \ref{lemma1}, that the stochastic
  convolution $Z(t)$ is mean square continuous. Hence it is sufficient
  to note that the same holds for the term
  $\int_0^t\TT(t-s)G(X(s))\,{\rm d} s$, that is $\mathbb{P}$-a.s. a
  continuous Bochner integral and then continuous as the composition
  of continuous functions on $[0,T]$.
\end{proof}

\begin{remark}
  By standard stopping time arguments one can actually show existence
  and uniqueness of a mild solution assuming only that $x_0$ is
  $\mathcal{F}_0$-measurable.
\end{remark}

In order to prove that the solution constructed above has c\`adl\`ag
paths, unfortunately one cannot adapt the factorization technique
developed for Wiener integrals (see e.g. \cite{DZ92}). However, the
c\`adl\`ag property of the solution was proved by Kotelenez
\cite{Kote-Doob}, under the assumption that $\AA$ is dissipative.
Therefore, thanks to proposition \ref{generation}, the solution
constructed above has c\`adl\`ag paths. One could also obtain this
property proving the following a priori estimate, which might be
interesting in its own right.

\begin{theorem}     \label{thm:esup}
  Under the assumptions of theorem \ref{th:caso-lip} the unique mild
  solution of problem \eqref{eq:NL} verifies
  \[
  \E \sup_{t\in[0,T]} |X(t)|_\HH^2 < \infty.
  \]
\end{theorem}

\begin{proof}
  Let us consider the It\^o formula for the function $|\cdot|_\HH^2$,
  applied to the process $X$. Although our computations are only
  formal, they can be justified using a classical approximation
  argument. We obtain
  \begin{equation*}
    {\rm d}|X(t)|_\HH^2 = 2 \langle X(t-), {\rm d}X(t) \rangle_\HH + {\rm d}[X](t).
  \end{equation*}
  By the dissipativity of the operator $\AA$ and the Lipschitz continuity
  of $G$, we obtain
  \begin{align*}
    \langle X(t-), {\rm d}X(t) \rangle_\HH &=
    \langle\AA X(t),X(t)\rangle_\HH\d
    t+\langle G(X(t)),X(t)\rangle_\HH\d t+\langle X(t-),\Sigma\d
    \LL(t)\rangle_\HH
    \\
    &\leq c_0|X(t)|_\HH^2+\langle X(t-),\Sigma\d \LL(t)\rangle_\HH.
  \end{align*}
Therefore
\[
|X(t)|_\HH^2 \leq |x_0|_\HH^2 + 2c_0\int_0^t|X(s)|_\HH^2\d s+2\int_0^t\langle
X(s-),\Sigma\d\LL(s)\rangle_\HH+\int_0^t |\Sigma|^2\d [\LL](s)
\]
and
\begin{align}\label{fava}
\E\sup_{t\leq T} |X(t)|_\HH^2 \leq &
\E|x_0|_\HH^2 + 2c_0T \E\sup_{t\leq T} |X(t)|_\HH^2\nonumber\\
& + 2 \E\sup_{t\leq T} \Big| \int_0^t\langle
X(s-),\Sigma\d\LL(s)\rangle_\HH\Big|
+ T \int_{\R^n} |\Sigma|^2 |x|^2\,\nu(dx),
\end{align}
where we have used the relation
\[
\E \sup_{t\leq T} [X](t) \leq \E\int_0^T |\Sigma|^2 \,d[\LL](t) =
\E\int_0^T |\Sigma|^2\,d\langle\LL\rangle(t) =
    T \int_{\R^n} \left|\Sigma
      \begin{pmatrix}
        0 \\ x
      \end{pmatrix}
    \right|^2 \,\nu(dx).
\]
By the Burkholder-Davis-Gundy inequality applied to
$M_t=\int_0^t\langle X(s-),\Sigma\d\LL(s)\rangle_\HH$,
there exists a constant $c_1$ such that
\begin{eqnarray}\label{mazza}
\E \sup_{t\leq T} \Big| \int_0^t\langle X(s-),\Sigma\d\LL(s)\rangle_\HH\Big|
&\leq& c_1\E\left(\left[\int_0^\cdot \langle X(s-),\Sigma\d\LL(s)\rangle_\HH
\right](T)\right)^{1/2}\nonumber\\
&\leq&c_1\E\left(\sup_{t\leq T}|X(t)|_\HH^2\int_0^T|\Sigma|^2\d[\LL](s)\right)^{1/2}\nonumber\\
&\leq& c_1\left(\varepsilon\E\sup_{t\leq T}|X(t)|_\HH^2+\frac{1}{4\varepsilon}
\E\int_0^T|\Sigma|^2\d[\LL](s)\right)\nonumber\\
&=& c_1\varepsilon\E\sup_{t\leq T}|X(t)|_\HH^2
+ \frac{c_1T}{4\varepsilon} \int_{\R^n}|\Sigma|^2|x|^2\nu(\d x),
\end{eqnarray}
where we have used the elementary inequality $ab\leq\varepsilon a^2+\frac{1}{4\varepsilon}b^2$.
Then by (\ref{fava}) and (\ref{mazza}) we have
\begin{eqnarray*}
\E\sup_{t\leq T} |X(t)|_\HH^2 &\leq&
\E|x_0|_\HH^2 
+ 2c_0T \E\sup_{t\leq T} |X(t)|_\HH^2
+ 2c_1\varepsilon\E\sup_{t\leq T}|X(t)|_\HH^2\\
&& + \left( \frac{c_1}{2\varepsilon}+1\right)T \int_{\R^n}|\Sigma|^2|x|^2\,\nu(dx),
\end{eqnarray*}
hence
\[
\E\sup_{t\leq T}|X(t)|_\HH^2 \leq
N
\left[\E|x_0|_\HH^2+T\left(1+\frac{c_1}{2\varepsilon}\right)\right]
<+\infty,
\]
where
\[
N = N(c_0,c_1,T,\varepsilon) = 
\frac{1}{1-2c_0T-2c_1\varepsilon}.
\]
Choosing $\varepsilon>0$ and $T>0$ such that $N<1$, one obtains the claim for a small time interval. The extension to arbitrary time interval follows by classical extension arguments.
\end{proof}

\subsection{FitzHugh-Nagumo type nonlinearity}
Let us now consider the general case of a nonlinear quasi-dissipative
drift term $\mathcal{F}$.
\begin{theorem} \label{nonlip}
  Let $\mathcal{F}: D(\mathcal{F})\subset \mathcal{H} \to \mathcal{H}$
  be defined as in (\ref{eq:F_fitz-nag}). Then the equation
  \begin{equation}\label{FHN}
    \left\{
      \begin{aligned}
        \d X(t)& = [\AA X(t) - \mathcal{F}(X(t))] \, \d t+ \Sigma \, {\rm
          d}\LL(t), \qquad t\in [0,T],\\
        X(0)&= x_0
      \end{aligned}
    \right.
  \end{equation}
  admits a unique mild solution, denoted by $X(t,x_0)$, which
  satisfies the estimate
  \[
  \E|X(t,x)-X(t,y)|^2 \leq e^{2\eta t}\E|x-y|^2.
  \]
  for all $x$, $y\in\HH$.
\end{theorem}
\begin{proof}
As observed in section \ref{sec:3} above, there exists $\eta>0$ such
that $F + \eta I$ is accretive. By a standard argument one can
reduce to the case of $\eta=0$ (see e.g. \cite{barbu}), which we
shall assume from now on, without loss of generality.
Let us set, for $\lambda>0$, $F_\lambda(u)=F((1+\lambda F)^{-1}(u))$
(Yosida regularization). $\FF_\lambda$ is then defined in the
obvious way.

Let $\mathcal{G}y = -\AA y + \FF(y)$. Then $\mathcal{G}$ is maximal monotone on $\HH$. In
fact, since $\AA$ is self-adjoint, setting
\[
\varphi(u) =
\begin{cases}
  |\AA^{1/2}u|^2, & u\in D(\AA^{1/2}) \\
  + \infty, & \text{otherwise},
\end{cases}
\]
one has $\AA=\partial\varphi$.
Let us also set $F=\partial g$, where $g:\R^m \to \R$ is a convex
function, the construction of which is straightforward. Well-known
results on convex integrals (see e.g. \cite[sec. 2.2]{barbu}) imply that
$F$ on $H$ is equivalently defined as $F=\partial I_g$, where
\[
I_g(u) =
\begin{cases}
  \displaystyle
  \int_{[0,1]^m} g(u(x))\,\d x, & \text{if \ } g(u) \in L^1([0,1]^m),\\
  + \infty, & \text{otherwise}.
\end{cases}
\]
Let us recall that
\[
\FF=\begin{pmatrix}-F\\0\end{pmatrix}.
\]
Since $D(\FF) \cap D(\AA)$ is not empty, $\mathcal{G}$ is maximal monotone if
$\varphi((I+\lambda\FF)^{-1}(u)) \leq \varphi(u)$ (see e.g. \cite[Thm.
9]{Bre-mm}), which is verified by a direct (but tedious) calculation
using the explicit form of $\AA$, since $(I+\lambda f_j)^{-1}$ is a
contraction on $\R$ for each $j=1,\ldots,m$.

Let us consider the regularized equation
\[
dX_\lambda(t) + \mathcal{G}_\lambda X_\lambda(t)\,dt = \Sigma\,d\LL(t).
\]
Appealing to It\^o's formula for the square of the norm one obtains
\[
|X_\lambda(t)|^2 + 2\int_0^t
\langle \mathcal{G}_\lambda X_\lambda(s),X_\lambda(s)\rangle\,ds
= |x_0|^2 + 2\int_0^t \langle X_\lambda(s-),\Sigma\,d\LL(s)\rangle
+ [X_\lambda](t)
\]
for all $t\in[0,T]$. Taking expectation on both sides yields
\begin{equation}     \label{eq:minchia}
\E|X_\lambda(t)|^2 + 2\E\int_0^t
\langle \mathcal{G}_\lambda X_\lambda(s),X_\lambda(s)\rangle\,ds
= |x_0|^2 + t \int_{\R^n} |\Sigma|^2\,|z|^2\,\nu(dz),
\end{equation}
where we have used the identity
\[
\E[X_\lambda](t) = \E\Big[ \int_0^\cdot \Sigma\,d\LL(s) \Big](t)
= t \int_{\R^n} |\Sigma|^2\,|z|^2\,\nu(dz).
\]
Since by (\ref{eq:minchia}) we have that $\{X_\lambda\}$ is a bounded
subset of $L^\infty([0,T],\elle^2(\HH))$, and $\elle^2(\HH)$ is
separable, Banach-Alaoglu's theorem implies that
\[
X_\lambda \stackrel{\ast}{\rightharpoonup} X
\qquad \text{in\ } L^\infty([0,T],\elle^2(\HH)),
\]
on a subsequence still denoted by $\lambda$.
Thanks to the assumptions on $f_j$, one can easily prove that $\langle
F(u),u\rangle\geq c|u|^{p+1}$ for some $c>0$ and $p\geq 1$, hence (\ref{eq:minchia})
also gives
\[
\E \int_0^T |X_\lambda(s)|_{p+1}^{p+1}\,ds < C,
\]
which implies that
\begin{equation}      \label{eq:cazzo}
X_\lambda \wto X \qquad \text{in\ }
L^{p+1}(\Omega\times[0,T]\times D,\mathbb{P} \times dt\times d\xi),
\end{equation}
where $D=[0,1]^m \times \R^n$.
Furthermore, (\ref{eq:minchia}) and (\ref{eq:cazzo}) also imply
\[
\mathcal{G}_\lambda X_\lambda \wto \eta
\qquad \text{in\ }
L^{\frac{p+1}{p}}(\Omega\times[0,T]\times D,\mathbb{P} \times dt\times d\xi).
\]
The above convergences immediately imply that $X$ and $\eta$ are
predictable, then in order to complete the proof of existence, we have to show that $\eta(\omega,t,\xi)=\mathcal{G}(X(\omega,t,\xi))$, $\mathbb{P}\times dt \times d\xi$-a.e..  
For this it is enough to show that
\[
\limsup_{\lambda \to 0} \E\int_0^T
\langle \mathcal{G}_\lambda X_\lambda(s),X_\lambda(s) \rangle\,ds \leq
\E\int_0^T
\langle \eta(s),X(s) \rangle\,ds.
\]
Using again It\^o's formula we get
\begin{equation}     \label{eq:tega}
\E|X(T)|^2 + 2\E\int_0^T
\langle \eta(s),X(s)\rangle\,ds
= |x_0|^2 + T \int_{\R^n} |\Sigma|^2\,|z|^2\,\nu(dz).
\end{equation}
However, (\ref{eq:cazzo}) implies that
\[
\liminf_{\lambda\to 0} \E|X_\lambda(T)|^2 \geq \E|X(T)|^2
\]
(see e.g. \cite[Prop. 3.5]{Bre-AF}), from which the claim follows
comparing (\ref{eq:minchia}) and (\ref{eq:tega}).

The Lipschitz dependence on the initial datum as well as (as a consequence)
uniqueness of the solution is proved by observing that $X(t,x)-X(t,y)$
satisfies $\mathbb{P}$-a.s. the deterministic equation
\[
\frac{d}{dt} (X(t,x)-X(t,y)) = \AA(X(t,x)-X(t,y)) - \FF(X(t,x))+\FF(X(t,y)),
\]
hence
\begin{eqnarray*}
\frac12 \frac{d}{dt} |X(t,x)-X(t,y)|^2 &=&
\big\langle \AA(X(t,x)-X(t,y)),X(t,x)-X(t,y)\big\rangle \\
&&- \big\langle \FF(X(t,x)-\FF(X(t,y)),X(t,x)-X(t,y)\big\rangle\\
&\leq& \eta|X(t,x)-X(t,y)|^2,
\end{eqnarray*}
where $X(\cdot,x)$ stands for the mild solution with initial datum $x$. By Gronwall's lemma we have
\[
\E|X(t,x)-X(t,y)|^2\leq e^{2\eta t}\E|x-y|^2,
\]
which concludes the proof of the theorem.
\end{proof}
\begin{remark}
  An alternative method to solve stochastic evolution equations with a
  dissipative nonlinear drift term is developed in \cite{DZ92,DZ96},
  for the case of Wiener noise, and in the recent book \cite{PZ-libro} for
  the case of L\'evy noise. This approach consists essentially in the
  reduction of the stochastic PDE to a deterministic PDE with random
  coefficients, by ``subtracting the stochastic convolution''. To
  carry out this plan one has to find a reflexive Banach space
  ${\mathcal V}$, continuously embedded in ${\mathcal H}$, which is
  large enough to contain the paths of the stochastic convolution, and
  at the same time not too large so that it is contained in the domain
  of the nonlinearity $\mathcal{F}$. In particular, in the case of
  equation (\ref{FHN}), theorem 10.14 in \cite{PZ-libro} yields existence
  and uniqueness of a mild solution provided, among other conditions, that
  \[
  \int_0^T |Z(t)|^{18}\,dt < \infty \qquad\mathbb{P}\text{-a.s.}
  \]
  The result could also be obtained applying theorem 10.15 of
  op.~cit., provided one can prove that $\mathcal{L}$ has c\`adl\`ag
  trajectories in the domain of a fractional power of a certain
  operator defined in terms of $\mathcal{A}$. In some specific cases,
  such condition is implied by suitable integrability conditions of
  the L\'evy measure. Unfortunately it seems to us rather difficult to
  verify such conditions, a task that we have not been able to
  accomplish. On the other hand, our approach, while perhaps less
  general, yields the well-posedness result under seemingly natural
  assumptions.
\end{remark}
\begin{remark}
  By arguments similar to those used in the proof of theorem
  \ref{thm:esup} one can also obtain that
  \[
  \E\sup_{t\leq T}|X_\lambda(t)|^2<C,
  \]
  i.e. that $\{X_\lambda\}$ is bounded in
  $\elle^2(L^\infty([0,T];\HH))$. By means of Banach-Alaoglu's
  theorem, one can only conclude that
  $X_\lambda\stackrel{\ast}{\rightharpoonup}X$ in
  $\elle^2(L^1([0,T];\HH))'$, which is larger than
  $\elle^2(L^\infty([0,T];\HH))$. In fact, from
  \cite[Thm. 8.20.3]{Edw-FA}, being $L^1([0,T];\HH)$ a separable
  Banach space, one can only prove that if $F$ is a continuous linear
  form on $\elle^2(L^1([0,T];\HH))$, then there exists a function $f$
  mapping $\Omega$ into $L^\infty([0,T];\HH)$ that is weakly
  measurable and such that
\[
F(g)=\E\langle f,g\rangle
\]
for each $g\in \elle^2(L^1([0,T];\HH))$.
\end{remark}

\subsection*{Acknowledgements}
The second named author is partly supported by the DFG through SFB
611 (Bonn) and by the EU through grant MOIF-CT-2006-040743.

\bibliographystyle{amsplain}
\bibliography{ref}

\def\polhk#1{\setbox0=\hbox{#1}{\ooalign{\hidewidth
  \lower1.5ex\hbox{`}\hidewidth\crcr\unhbox0}}}
\providecommand{\bysame}{\leavevmode\hbox to3em{\hrulefill}\thinspace}
\providecommand{\MR}{\relax\ifhmode\unskip\space\fi MR }
\providecommand{\MRhref}[2]{%
  \href{http://www.ams.org/mathscinet-getitem?mr=#1}{#2}
}
\providecommand{\href}[2]{#2}
\begin{thebibliography}{10}

\bibitem{AlbRud-LI}
S.~Albeverio and B.~R{\"u}diger, \emph{Stochastic integrals and the
  {L}\'evy-{I}to decomposition theorem on separable {B}anach spaces}, Stoch.
  Anal. Appl. \textbf{23} (2005), no.~2, 217--253. \MR{MR2130348 (2008e:60157)}

\bibitem{App2}
D.~Applebaum, \emph{L\'evy processes and stochastic calculus}, Cambridge
  University Press, Cambridge, 2004. \MR{MR2072890 (2005h:60003)}

\bibitem{App1}
\bysame, \emph{Martingale-valued measures, {O}rnstein-{U}hlenbeck processes
  with jumps and operator self-decomposability in {H}ilbert space}, S\'eminaire
  de Probabilit\'es XXXIX, Lecture Notes in Math., vol. 1874, Springer, Berlin,
  2006, pp.~171--196. \MR{MR2276896 (2008d:60062)}

\bibitem{barbu}
V.~Barbu, \emph{Analysis and control of nonlinear infinite-dimensional
  systems}, Academic Press Inc., Boston, MA, 1993. \MR{MR1195128 (93j:49002)}

\bibitem{BoZi}
S.~Bonaccorsi and G.~Ziglio, \emph{A semigroup approach to stochastic dynamical
  boundary value problems}, Systems, control, modeling and optimization, IFIP
  Int. Fed. Inf. Process., vol. 202, Springer, New York, 2006, pp.~55--65.

\bibitem{Bre-mm}
H.~Br{\'e}zis, \emph{Monotonicity methods in {H}ilbert spaces and some
  applications to nonlinear partial differential equations}, Contributions to
  nonlinear functional analysis (Proc. Sympos., Math. Res. Center, Univ.
  Wisconsin, Madison, Wis., 1971), Academic Press, New York, 1971,
  pp.~101--156. \MR{MR0394323 (52 \#15126)}

\bibitem{Bre-AF}
\bysame, \emph{Analyse fonctionnelle}, Masson, Paris, 1983. \MR{MR697382
  (85a:46001)}

\bibitem{cerrai-libro}
S.~Cerrai, \emph{Second order {PDE}'s in finite and infinite dimension},
  Lecture Notes in Mathematics, vol. 1762, Springer-Verlag, Berlin, 2001.
  \MR{2002j:35327}

\bibitem{choj}
A.~Chojnowska-Michalik, \emph{On processes of {O}rnstein-{U}hlenbeck type in
  {H}ilbert space}, Stochastics \textbf{21} (1987), no.~3, 251--286.

\bibitem{DP-K}
G.~Da~Prato, \emph{Kolmogorov equations for stochastic {PDE}s}, Birkh\"auser
  Verlag, Basel, 2004. \MR{MR2111320 (2005m:60002)}

\bibitem{DZ92}
G.~Da~Prato and J.~Zabczyk, \emph{Stochastic equations in infinite dimensions},
  Cambridge University Press, Cambridge, 1992. \MR{MR1207136 (95g:60073)}

\bibitem{DZ96}
\bysame, \emph{Ergodicity for infinite-dimensional systems}, Cambridge
  University Press, Cambridge, 1996. \MR{MR1417491 (97k:60165)}

\bibitem{Edw-FA}
R.~E. Edwards, \emph{Functional analysis. {T}heory and applications}, Holt,
  Rinehart and Winston, New York, 1965. \MR{MR0221256 (36 \#4308)}

\bibitem{GS-III}
I.~I. Gihman and A.~V. Skorohod, \emph{The theory of stochastic processes.
  {III}}, Springer-Verlag, Berlin, 1979.

\bibitem{Hau}
E.~Hausenblas, \emph{Existence, uniqueness and regularity of parabolic {SPDE}s
  driven by {P}oisson random measure}, Electron. J. Probab. \textbf{10} (2005),
  1496--1546 (electronic).

\bibitem{HoHu}
A.L. Hodgkin and A.F. Huxley, \emph{A quantitative description of membrane
  current and its application to conduction and excitation in nerve}, J.
  Physiol. \textbf{117} (1952), no.~2, 500--544.

\bibitem{KaWo}
G.~Kallianpur and R.~Wolpert, \emph{Infinite-dimensional stochastic
  differential equation models for spatially distributed neurons}, Appl. Math.
  Optim. \textbf{12} (1984), no.~2, 125--172.

\bibitem{KaXi}
G.~Kallianpur and J.~Xiong, \emph{Diffusion approximation of nuclear
  space-valued stochastic-differential equations driven by {P}oisson random
  measures}, Ann. Appl. Probab. \textbf{5} (1995), no.~2, 493--517.

\bibitem{KeeSne}
J.~Keener and J.~Sneyd, \emph{Mathematical physiology}, Springer, New York,
  1998.

\bibitem{Kote-Doob}
P.~Kotelenez, \emph{A stopped {D}oob inequality for stochastic convolution
  integrals and stochastic evolution equations}, Stochastic Anal. Appl.
  \textbf{2} (1984), no.~3, 245--265.

\bibitem{KrMuSi}
M.~Kramar~Fijav{\v{z}}, D.~Mugnolo, and E.~Sikolya, \emph{Variational and
  semigroup methods for waves and diffusion in networks}, Appl. Math. Optim.
  \textbf{55} (2007), no.~2, 219--240.

\bibitem{LiMa}
J.-L. Lions and E.~Magenes, \emph{Non-homogeneous boundary value problems and
  applications. {V}ol. {I}}, Springer-Verlag, New York, 1972.

\bibitem{Lu}
A.~Lunardi, \emph{Analytic semigroups and optimal regularity in parabolic
  problems}, Birkh\"auser Verlag, Basel, 1995. \MR{MR1329547 (96e:47039)}

\bibitem{MuRo}
D.~Mugnolo and S.~Romanelli, \emph{Dynamic and generalized {W}entzell node
  conditions for network equations}, Math. Methods Appl. Sci. \textbf{30}
  (2007), no.~6, 681--706.

\bibitem{ouhabaz}
E.~M. Ouhabaz, \emph{Analysis of heat equations on domains}, Princeton
  University Press, Princeton, NJ, 2005.

\bibitem{PZ-libro}
Sz. Peszat and J.~Zabczyk, \emph{Stochastic partial differential equations with
  {L}\'evy noise}, Cambridge University Press, Cambridge, 2007. \MR{MR2356959}

\bibitem{Giugi}
C.~Roc{\c{s}}oreanu, A.~Georgescu, and N.~Giurgi{\c{t}}eanu, \emph{The
  {F}itz{H}ugh-{N}agumo model}, Kluwer Academic Publishers, Dordrecht, 2000.

\bibitem{Walsh}
J.~B. Walsh, \emph{An introduction to stochastic partial differential
  equations}, \'Ecole d'\'et\'e de probabilit\'es de Saint-Flour, XIV---1984,
  Lecture Notes in Math., vol. 1180, Springer, Berlin, 1986, pp.~265--439.

\end{thebibliography}

\end{document}